\documentclass[12pt]{amsart}
\usepackage[english]{babel}
\usepackage[utf8]{inputenc}

\usepackage{amssymb}

\usepackage[colorlinks=true]{hyperref}
\usepackage{paralist}
\usepackage{cleveref}

\usepackage{thm-restate}

\usepackage{tikz}
\usetikzlibrary{calc}

\usepackage[a4paper,top=3.3cm,bottom=3cm,left=3cm,right=3cm,bindingoffset=5mm]{geometry}

\usepackage{color}

\newtheorem{theorem}{Theorem}[section]
\newtheorem{lemma}[theorem]{Lemma}
\newtheorem{corollary}[theorem]{Corollary}
\newtheorem{proposition}[theorem]{Proposition}

\theoremstyle{definition}

\newtheorem{remark}[theorem]{Remark}
\newtheorem{question}[theorem]{Question}
\newtheorem{example}[theorem]{Example}

\newtheorem*{notation}{Notation}

\newcommand{\af}{\mathfrak{a}}

\newcommand{\KK}{\mathbb{K}}
\newcommand{\NN}{\mathbb{N}}
\newcommand{\ZZ}{\mathbb{Z}}
\newcommand{\QQ}{\mathbb{Q}}
\newcommand{\FF}{\mathbb{F}}

\newcommand{\set}[1]{\{#1\}}
\newcommand{\with}{\,\colon\,}
\newcommand{\restr}[2]{{#1}|_{#2}}

\DeclareMathOperator{\Tor}{Tor}
\DeclareMathOperator{\Char}{char}

\DeclareMathOperator{\reg}{reg}

\newcommand{\Hi}[2]{\tilde{H}_{#1}({#2};\KK)}

\let\phi\varphi

\begin{document}

\title[Golod property and Characteristic]{The Golod property for Stanley-Reisner rings in varying characteristic}

\author{Lukas Katth\"an}

\address{Universit\"at Osnabr\"uck, FB Mathematik/Informatik, 49069 Osnabr\"uck, Germany}
\email{lukas.katthaen@uos.de}



\begin{abstract}
	We show that the Golod property of a Stanley-Reisner ring can depend on the characteristic of the base field.
	More precisely, for every finite set $T$ of prime numbers we construct simplicial complexes $\Delta$ and $\Gamma$, such that $\KK[\Delta]$ is Golod exactly in the characteristics in $T$ and $\KK[\Gamma]$ is Golod exactly in the characteristics not in $T$.
	
	Along the way, we show that a one-dimensional simplicial complex is Golod if and only if it is chordal.
\end{abstract}

\maketitle

\section{Introduction}
Let $\KK$ be a field and $S = \KK[X_1, \dotsc, X_n]$ be the polynomial ring, endowed with the standard $\ZZ$-grading. 
For a homogeneous ideal $\af$, let $A := S/\af$  be the quotient ring.
As $A$ is a standard-graded algebra, its $\Tor$-algebra $\Tor^A_*(\KK,\KK)$ inherits a $\ZZ$-grading.
The \emph{Poincar\'e-series} of $A$ is the formal power series
\[ P_A(t,x) = \sum_{i, j \geq 0} (\dim_\KK \Tor^A_i(\KK,\KK)_j) t^i x^j, \]
where $\Tor^A_i(\KK,\KK)_j$ denotes the homogeneous component of $\Tor^A_i(\KK,\KK)$ in degree $j$.
The algebra $A$ is called \emph{Golod} if the following holds:
\begin{equation}\label{eq:ps}
	P_A(t, x) = \frac{(1+t x)^n}{1-\sum_{i\geq 1}\sum_{j \geq 0} \dim_\KK (\Tor^S_i(A,\KK)_j) t^{i+1} x^j}
\end{equation}
In general, $P_A(t, x)$ is componentwise bounded above by the right-hand side of \eqref{eq:ps}, as it was shown by Serre.

Golod algebras are surprisingly common.
For example, it has been proven by Herzog and Huneke \cite{HH} that if $\af \subseteq S$ is a homogeneous ideal, then $S/\af^k$ is Golod for every $k > 1$.
Further, Herzog, Welker and Reiner showed in \cite{HRW} that $S/\af$ is Golod if $\af$ has a componentwise linear resolution.
We refer the reader to the survey article \cite{Avra} by Avramov for more information on Golod algebras.

In \cite{Jo06} and \cite{BJ07} Berglund and J\"ollenbeck considered the Golod property for Stanley-Reisner rings.
They give a combinatorial characterization of Golodness in the class of flag simplicial complexes,
which in particular implies that the Golod property of these complexes does not depend on the field $\KK$ of coefficients.
Thus it seems natural to ask whether one can find a combinatorial description of the Golod property for Stanley-Reisner rings of general simplicial complexes.
The general expectation seems to be that this is not the case, i.e. for sufficiently complicated complexes the Golod property might depend $\KK$.
However, no example of this phenomenon was known.

In the present note, we provide a construction for such examples.
More precisely, we prove the following:
\begin{restatable*}{theorem}{mainthm}\label{thm:main}
	Let $T$ be a finite set of prime numbers.
	\begin{enumerate}
		\item There exists a simplicial complex $\Delta$ such that $\KK[\Delta]$ is Golod if and only if $\Char \KK \in T$.
		\item Also, there exists a simplicial complex $\Gamma$ such that $\KK[\Gamma]$ is Golod if and only if $\Char \KK \notin T$.
	\end{enumerate}
\end{restatable*}
We remark that for many properties of $\KK[\Delta]$, for example the property of having a componentwise linear resolution or being Cohen-Macaulay, only the second case can happen.
Our proof is constructive, but for convenience we also give two explicit examples for the case $T = \set{2}$ in Section \ref{sec:example}.

Let us explain the mechanisms which cause these two cases.
By a result of Iriye and Kishimoto \cite{IK}, the Golod property depends on the vanishing of certain maps between homology groups, see \Cref{prop:IK}.
On the one hand, it might happen that these homology groups are torsion groups and thus vanish in all but finitely many characteristics.
On the other hand, a map between the free parts of the homology groups might be the multiplication by some number $N$.
In this case, the map vanishes exactly for the finitely many prime divisors of $N$.

Given our main result, one might be tempted to ask if the finiteness assumption is necessary.
In other words, one could ask if there exists a simplicial complex which is Golod in infinitely many characteristics,
and non-Golod in infinitely many other characteristics.
Our second  result gives a negative answer to this question. Thus, the complexes constructed in \Cref{thm:main} are ``worst possible''.

\begin{restatable*}{proposition}{finiteness}\label{prop:finite}
For a simplicial complex $\Delta$, the following holds:
\begin{enumerate}
	\item The Golod property of $\KK[\Delta]$ depends only on the characteristic of $\KK$. More precisely, if $\KK$ and $\KK'$ are two fields with the same characteristic, then $\KK[\Delta]$ is Golod if and only if $\KK'[\Delta]$ is Golod.
	\item If $\QQ[\Delta]$ is Golod, then $\FF_p[\Delta]$ is Golod for all but at most finitely many primes $p$.
	\item If $\QQ[\Delta]$ is not Golod, then $\FF_p[\Delta]$ is Golod for at most finitely many primes $p$.
\end{enumerate}
Here, $\FF_p$ denotes the field with $p$ elements.
\end{restatable*}

This paper is organized as follows.
In \Cref{sec:prelim}, we collect some background information on Golod rings.
In particular, we derive some useful criteria for deciding whether a Stanley-Reisner ring is Golod.
In the third section, we first consider the contributions of the one-skeleton of a simplicial complex to the Koszul homology.
As a by-product, we show that a one-dimensional simplicial complex is Golod if and only if it is chordal.
After that we prove \Cref{thm:main} and \Cref{prop:finite}.
In \Cref{sec:example}, two explicit examples of Stanley-Reisner rings whose Golodness depends on the field are provided.
Finally, in the last section we shortly discuss a relation to \emph{decomposition $k$-chordal} complexes \cite{ANS} and pose a question for an improved criterion for Golodness.

\section{Preliminaries about the Golod property for Stanley-Reisner rings}\label{sec:prelim}
In this section, we recall some facts about Golod Stanley-Reisner rings.
We refer the reader to \cite{AvraG} and \cite{Avra} for a comprehensive treatment of general Golod rings.

Let $\Delta$ be a fixed simplicial complex.
We denote its set of vertices by $V(\Delta)$ and set $n := \#V(\Delta)$.
Let $\KK[\Delta] = S/I_\Delta$ its Stanley-Reisner ring (over some fixed field $\KK$), cf. Chapter 5 of \cite{BH}.
Here, $S = \KK[X_v \with v \in V(\Delta)]$ is a polynomial ring and 
\[ I_\Delta = \left(\prod_{v \in M} X_v \with M \subseteq V(\Delta), M \notin \Delta \right) \]
is the Stanley-Reisner ideal of $\Delta$.
$\KK[\Delta]$ carries a natural $\NN^n$-grading and we will occasionally identify squarefree multidegrees with subsets of $[n] := \set{1,\dotsc,n}$.

Let further $K_{\KK[\Delta]}$ denote the \emph{Koszul complex} of $\KK[\Delta]$, see \cite[Chapter 1.6]{BH}.
It carries natural ``homological'' $\NN$-grading, in addition to the ``internal'' $\NN^n$-grading inherited from $\KK[\Delta]$.
For a homogeneous element $a \in K_{\KK[\Delta]}$, we denote by $|a|$ its homological degree and by $\deg a$ its internal multidegree.
It is well-known that $K_{\KK[\Delta]}$ is an $S$-algebra, which is skew-commutative with respect to the homological grading.
Also, the multiplication on $K_{\KK[\Delta]}$ induces a multiplication on its homology $H_*(K_{\KK[\Delta]}) = \Tor_*^S({\KK[\Delta]}, \KK)$.

By definition, $\KK[\Delta]$ is called Golod if the equality \eqref{eq:ps} of power series holds.
Golod \cite{golod} showed that this is equivalent to the condition that all Massey products on the homology $H_*(K_{\KK[\Delta]})$ are trivial.
We will recall the definition of the latter below in the proof of \Cref{lem:2dim}.
For now, we only point out that the second Massey product of two elements $a_1, a_2 \in H_*(K_{\KK[\Delta]})$ is just the usual product.

Hochster \cite{hochster} gave a topological interpretation of the homogeneous strands of the Koszul complex $K_{\KK[\Delta]}$.
Namely, the strand $(K_{\KK[\Delta]})_I$ in degree $I \subseteq [n]$ is isomorphic to the simplicial cochain complex on the \emph{restriction} $\restr{\Delta}{I}$, where the cohomological grading on the latter is reversed and shifted, such that the vertices of $\restr{\Delta}{I}$ sit in degree $\#I - 1$. 
In particular, passing to homology yields Hochster's formula
\begin{equation}\label{eq:hochster}
	H_i(K_{\KK[\Delta]})_I = \tilde{H}^{\#I - 1 - i}(\restr{\Delta}{I};\KK).
\end{equation}
This equation yields a useful interpretation of the product on $H_*(K_{\KK[\Delta]})$, cf.~\cite[Proposition 3.2.10]{tt}.
Indeed, the latter is in fact induced from the inclusions $\restr{\Delta}{I \cup J} \hookrightarrow \restr{\Delta}{I} * \restr{\Delta}{J}$ for all $I,J \neq \emptyset$ with $I \cap J =\emptyset$, where
$\Delta * \Gamma := \set{F \cup G \with F \in \Delta, G\in \Gamma}$ denotes the \emph{join}.
Since we are working over a field, we may replace cohomology by homology.
This leads to the following useful criterion for the vanishing of the product:
\begin{proposition}[Proposition 6.3, \cite{IK}]\label{prop:IK}
	The product on $H_*(K_{\KK[\Delta]})$ is trivial if and only if the inclusion $\restr{\Delta}{I \cup J} \hookrightarrow \restr{\Delta}{I} * \restr{\Delta}{J}$ induces the zero map in homology with coefficients in $\KK$ for all $\emptyset \neq I,J \subset V(\Delta)$ with $I \cap J = \emptyset$.
\end{proposition}
\begin{notation}
	Let $i \in \NN$ and let $I, J$ be two non-empty disjoint subsets of $V(\Delta)$.
	We write $\phi_i^{I,J}\colon \Hi{i}{\restr{\Delta}{I \cup J}} \rightarrow \Hi{i}{\restr{\Delta}{I} * \restr{\Delta}{J}}$ for the map induced by the inclusion $\restr{\Delta}{I \cup J} \hookrightarrow \restr{\Delta}{I} * \restr{\Delta}{J}$.
\end{notation}
\noindent The following is just a reformulation of \Cref{prop:IK}:
\begin{corollary}\label{prop:hom}
	The product on $H_*(K_{\KK[\Delta]})$ is trivial if and only if $\phi_i^{I,J} = 0$ for all $i \in \NN$ and all nonempty disjoint $I,J \subset V(\Delta)$.
\end{corollary}

The following lemma allows us to concentrate on the product on $H_*(K_{\KK[\Delta]})$, so we do not need to consider the higher Massey products. 
\begin{lemma}\label{lem:2dim}
	Assume that $\dim \Delta \leq 2$ and
	that for any two disjoint non-empty sets $I,J \subset V(\Delta)$, at least one of the two complexes $\restr{\Delta}{I}$ and $\restr{\Delta}{J}$ is connected.
	Then $\KK[\Delta]$ is Golod if and only if the product on $H_*(K_{\KK[\Delta]})$ is trivial.
\end{lemma}
\begin{proof}
The necessity is clear, so we only consider the sufficiency.

Let us recall the definition of the Massey products.
As mentioned above, the second Massey product of two elements $a_1, a_2 \in H_*(K_{\KK[\Delta]})$ is just the usual product.
Let us denote it by $\mu_2(a_1,a_2)$.
For $n \geq 3$, the $n$-th Massey product is a partially defined set-valued function, which assigns to $n$ elements $a_1, \dotsc, a_n \in H_*(K_{\KK[\Delta]})$ a set $\mu_n(a_1, \dotsc, a_n) \subset H_*(K_{\KK[\Delta]})$.
It is defined if there exist elements $a_{ij} \in K_{\KK[\Delta]}$ for $1 \leq i \leq j \leq n$, such that $da_{ii} = 0$, $[a_{ii}] = a_i$ and 
\[ da_{ij} = \sum_{v=i}^j \bar{a}_{iv}a_{vj}, \]
where $\bar{a} = (-1)^{|a|+1}$. 
Then $\sum_{v=1}^n \bar{a}_{iv}a_{vj}$ is called a Massey product of $a_1, \dotsc, a_n$ and $\mu_n(a_1, \dotsc, a_n)$ is the set of all these elements.
One says that the Massey product is \emph{trivial up to $n$}, if for all $2\leq i \leq n$ and all for all $i$-tuples $a_1, \dotsc, a_i \in H_*(K_{\KK[\Delta]})$, the set $\mu_i(a_1, \dotsc, a_i)$ is defined and contains only zero.
Finally, we say that $H_*(K_{\KK[\Delta]})$ has \emph{trivial Massey products} if it is trivial for all $n$.
We will use the following properties of the Massey products, which follow from the definition and are well-known: 

\begin{enumerate}[(a)]
	\item If the Massey product is trivial up to $n-1$, then $\mu_{n}(a_1, \dotsc, a_{n})$ is defined for all $n$-tuples $a_1, \dotsc, a_{n}$ and contains only one element.
		So, in this case one has an actual map
		$ \mu_{n}: H_*(K_{\KK[\Delta]})^{\otimes {n}} \to H_*(K_{\KK[\Delta]})$.
	\item If $\mu_n(a_1, \dotsc, a_n)$ is defined, then every element in this set has multidegree $\sum_i\deg a_i$ and homological degree $\sum_i (|a_i|+1) - 2$.
	\item If $\mu_n(a_1, \dotsc, a_n)$ is defined and one of the $a_i$ is zero, then $\mu_n(a_1, \dotsc, a_n)$ contains zero.
\end{enumerate}

By induction, assume that the Massey product is trivial up to $n-1$ for some $n \geq 3$, so $\mu_{n}$ is a map by (a).
Then (b) implies that $\mu_n$ is zero in any non-squarefree degree.
In a squarefree multidegree $I$, Hochster's formula \eqref{eq:hochster} implies that $\mu_n$ decomposes into a direct sum of maps of the form
\begin{equation}\label{eq:degree}
	\bigoplus_{i_1, \dotsc, i_n} \Hi{i_1}{\restr{\Delta}{I_1}} \otimes \dotsb \otimes \Hi{i_n}{\restr{\Delta}{I_n}} \to \Hi{i}{\restr{\Delta}{I}} 
\end{equation}
where $I_1 \cup I_2 \cup \dotsb \cup I_n = I$ is a disjoint decomposition of $I$ with $I_j \neq \emptyset$ for all $j$, 
and the sum runs over all $0\leq i_1, \dotsc, i_n \leq n$ with $\sum_j i_j = i-1$.
As $\dim \Delta = 2$, it follows from \eqref{eq:degree} and (c) that for $\mu_n$ being not trivial, 
there have to exist non-empty disjoint subsets $I_1, \dotsc, I_n \subset V$, such that either 
\begin{itemize}
	\item $\Hi{0}{\restr{\Delta}{I_i}} \neq 0$ for all $i$, or
	\item $\Hi{0}{\restr{\Delta}{I_i}} \neq 0$ for all $i$ but one, say $i_1$, and $\Hi{1}{\restr{\Delta}{I_{i_1}}} \neq 0$.
\end{itemize}
In both cases, there are at least two disjoint non-empty subsets $I,J \subset V$ such that both $\restr{\Delta}{I}$ and $\restr{\Delta}{J}$ are disconnected, contradicting our assumption.
\end{proof}

\begin{remark}
\begin{asparaenum}
	\item Berglund and J\"ollenbeck showed in \cite[Theorem 5.1]{BJ07} that in general $\KK[\Delta]$ is Golod 
	if and only if the product on $H_*(K_{\KK[\Delta]})$ is trivial, so the assumptions of \Cref{lem:2dim} are superfluous.
	However, we prefer to prove our main result independently of that result. 
	The reason is that \Cref{lem:2dim} has a very simple proof and is sufficient for our present purpose,
	while the proof of \cite[Theorem 5.1]{BJ07} is rather long and involved.
	\item It is possible (and perhaps preferable) to consider Massey products in the language of $A_\infty$-algebras, cf. \cite{LPWZ, val}. 
	This allows one to replace the classical Massey products with maps $\mu_n: H_*(K_{\KK[\Delta]})^{\otimes n} \to H_*(K_{\KK[\Delta]})$ which are always defined and no longer set-valued.
	However, for the purpose of the present paper the classical approach to Massey products is sufficient.
\end{asparaenum}
\end{remark}

\section{Proof of the main results}
In this section, we first give some useful results concerning the maps $\phi_i^{I,J}$. 
After these preparations, our main result \Cref{thm:main} is be proven.
In the last part of the section, we prove the above mentioned \Cref{prop:finite}

We start with some notation.
For a simplicial complex $\Delta$, we write $C_i(\Delta; \KK)$ for the vector space of simplicial $i$-chains of $\Delta$.
We say that an $i$-cycle $c \in C_i(\Delta, \KK)$ \emph{contains} an $i$-face $F$ if it has a non-zero coefficient in $c$.
An $i$-cycle $c \in C_i(\Delta, \KK)$ is called \emph{complete} \cite{ANS} if it is the boundary of a $(i+1)$-simplex (which does not need to be a face of $\Delta$).
The following is a simple criterion to decide the vanishing or non-vanishing of $\phi_i^{I,J}$ in some cases.
\begin{lemma}\label{lem:complete}
	Let $\Delta$ be a simplicial complex, let $\emptyset \neq I,J \subseteq V(\Delta)$ with $I \cap J = \emptyset$ and let $c \in C_i(\restr{\Delta}{I \cup J};\KK)$ be an $i$-cycle, for some $i$.
	\begin{enumerate}
		\item If $c$ is complete, then $\phi_i^{I,J}(c) = 0$.
		\item If $c$ contains an $i$-face $F$, such that $F \cap I$ and $F \cap J$ are both facets of $\restr{\Delta}{I}$ and $\restr{\Delta}{J}$, respectively, then $\phi_i^{I,J}(c) \neq 0$.
	\end{enumerate}
\end{lemma}
\begin{proof}
\begin{asparaenum}
\item First, assume that some vertices of $c$ are in $I$ and some are in $J$, i.e., there are non-empty faces $\sigma \in \restr{\Delta}{I}, \tau \in \restr{\Delta}{J}$ such that the support of $c$ equals $\sigma \cup \tau$.
It holds that $\phi_i^{I,J}(c) = \partial(\sigma \cup \tau)$ because $c$ is complete.
Hence $\phi_i^{I,J}(c)$ is a boundary.

Next, consider the case that all vertices of $c$ are in one of the sets, say $I$.
Then $c$ is the boundary of the cone $c * \set{w}$ for any vertex $w \in J$.

\item Under the given hypothesis $F$ is a facet of $\restr{\Delta}{I} * \restr{\Delta}{J}$. Hence $\phi_i^{I,J}(c)$ is not a boundary.
\end{asparaenum}
\end{proof}

It was proven in \cite[Proposition 6.4]{BJ07} that the $1$-skeleton of a Golod simplicial complex is chordal.
The following extends this result.
\begin{proposition}\label{prop:chordal}
For a simplicial complex $\Delta$, the following are equivalent:
\begin{enumerate}
	\item The $1$-skeleton of $\Delta$ is a chordal graph.
	\item $\phi_1^{I,J} = 0$ for all non-empty disjoint subsets $I,J \subset V(\Delta)$.
\end{enumerate}
\end{proposition}
\begin{proof}
\begin{asparaenum}
\item[(1) $\Rightarrow$ (2)]
	If the $1$-skeleton of $\Delta$ is a chordal graph, then every $1$-cycle can be written as a sum of ``triangles''.
	But a triangle is a complete $1$-cycle, so the claim follows form \Cref{lem:complete}.

\item[(2) $\Rightarrow$ (1)]
	Assume that the $1$-skeleton of $\Delta$ contains a chordless cycle $c$ of length at least $4$.
	Let $v$ be any vertex of $c$ and let $w_1, w_2$ denote its two neighbors.
	Set $I := \set{w_1,w_2}$ and let $J$ be the set of all other vertices of $c$.
	Note that the vertices $v$ and $w_1$ are isolated in $\restr{\Delta}{I}$ and $\restr{\Delta}{J}$, respectively, because $c$ is chordless.
	
	But the homology class $\tilde{c} \in \Hi{1}{\restr{\Delta}{I \cup J}}$ corresponding to $c$ contains the edge $\set{v,w_1}$, hence $\phi_1^{I,J}(\tilde{c}) \neq 0$ by \Cref{lem:complete}.
\end{asparaenum}
\end{proof}

It is known that a flag simplicial complex is Golod if and only if its one-skeleton is chordal, see \cite[Theorem 6.7]{BJ07}.
A one-dimensional simplicial complex does not need to be flag, so the following can be seen as a partial extension of that result.
\begin{corollary}\label{cor:1}
Let $\Delta$ be a one-dimensional simplicial complex.
Then $\KK[\Delta]$ is Golod if and only if $\Delta$ is chordal.
\end{corollary}
After the preparation of the present article I learned that Iriye and Kishimoto proved a stronger version of this corollary in \cite[Theorem 11.8]{IK}.
\begin{proof}
This is immediate from the foregoing \Cref{prop:chordal} and \cite[Theorem 5.1]{BJ07}.
Alternatively, if $\Delta$ is chordal, then its Stanley-Reisner ideal is componentwise linear (the degree $2$ part is linear by \cite{fro} and the regularity is at most $3$), so $\KK[\Delta]$ is Golod, cf.~\cite{HRW}.
\end{proof}

Now we turn to the proof of our  main result:
\mainthm
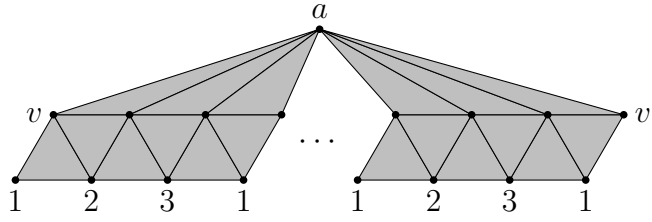
\begin{figure}[t]
\begin{tikzpicture}[scale=1]
\coordinate (1) at (0,0);
\coordinate (2) at (1,0);
\coordinate (3) at (2,0);
\coordinate (12) at (3,0);

\coordinate (1N) at (4.5,0);
\coordinate (2N) at ($(1N)+(1,0)$);
\coordinate (3N) at ($(1N)+(2,0)$);
\coordinate (1Np) at ($(1N)+(3,0)$);

\coordinate (v11) at (60:1);
\coordinate (v21) at ($(v11)+(1,0)$);
\coordinate (v31) at ($(v11)+(2,0)$);
\coordinate (v12) at ($(v11)+(3,0)$);

\coordinate (v1N) at ($(v11)+(4.5,0)$);
\coordinate (v2N) at ($(v1N)+(1,0)$);
\coordinate (v3N) at ($(v1N)+(2,0)$);
\coordinate (v1Np) at ($(v1N)+(3,0)$);

\coordinate (a) at (4, 2);


\filldraw[fill=gray!50] (1) -- (2) -- (v11) -- cycle;
\filldraw[fill=gray!50] (2) -- (3) -- (v21) -- cycle;
\filldraw[fill=gray!50] (3) -- (12) -- (v31) -- cycle;
\filldraw[fill=gray!50] (1N) -- (2N) -- (v1N) -- cycle;
\filldraw[fill=gray!50] (2N) -- (3N) -- (v2N) -- cycle;
\filldraw[fill=gray!50] (3N) -- (1Np) -- (v3N) -- cycle;


\filldraw[fill=gray!50] (v11) -- (2) -- (v21) -- cycle;
\filldraw[fill=gray!50] (v21) -- (3) -- (v31) -- cycle;
\filldraw[fill=gray!50] (v31) -- (12) -- (v12) -- cycle;
\filldraw[fill=gray!50] (v1N) -- (v2N) -- (2N) -- cycle;
\filldraw[fill=gray!50] (v2N) -- (v3N) -- (3N) -- cycle;
\filldraw[fill=gray!50] (v3N) -- (v1Np) -- (1Np) -- cycle;


\filldraw[fill=gray!50] (v11) -- (v21) -- (a) -- cycle;
\filldraw[fill=gray!50] (v21) -- (v31) -- (a) -- cycle;
\filldraw[fill=gray!50] (v31) -- (v12) -- (a) -- cycle;
\filldraw[fill=gray!50] (v1N) -- (v2N) -- (a) -- cycle;
\filldraw[fill=gray!50] (v2N) -- (v3N) -- (a) -- cycle;
\filldraw[fill=gray!50] (v3N) -- (v1Np) -- (a) -- cycle;

\fill (1) circle (0.05) node[anchor=north] {$1$};
\fill (2) circle (0.05) node[anchor=north] {$2$};
\fill (3) circle (0.05) node[anchor=north] {$3$};
\fill (12) circle (0.05) node[anchor=north] {$1$};
\fill (1N) circle (0.05) node[anchor=north] {$1$};
\fill (2N) circle (0.05) node[anchor=north] {$2$};
\fill (3N) circle (0.05) node[anchor=north] {$3$};
\fill (1Np) circle (0.05) node[anchor=north] {$1$};

\foreach \a in {(v11),(v21),(v31),(v12),(v1N),(v2N),(v3N),(v1Np)}
\fill \a circle (0.05) node[anchor=south] {};

\node at (4, 0.5) {$\cdots$};
\fill (a) circle (0.05) node[anchor=south] {$a$};
\draw (v11) node[anchor=east]  {$v$};
\draw (v1Np) node[anchor=west] {$v$};

\end{tikzpicture}
\caption{The complex $\Delta_1$ used in the proof of \Cref{thm:main}. The sequence $1,2,3$ repeats $N$ times.}
\label{fig:ex1}
\end{figure}
\begin{proof}
	Let $N$ be the product of the elements of $T$.
	
	Let $\Delta_1$ be the simplicial complex indicated in Fig. \ref{fig:ex1}, where vertices with the same label are to be identified.
	Here, the sequence $1,2,3$ of the bottom vertices is repeated $N$ times.
	So, topologically $\Delta_1$ is obtained by gluing the boundary of a $2$-cell $N$ times around the cycle $\gamma$ formed by the vertices $1,2$ and $3$.
	
	\medskip
	(1) Let $\Delta_2$ be the complex obtained by gluing two copies of $\Delta_1$ along $\gamma$.
	We denote the vertex ``$a$'' in the second copy of $\Delta_1$ by $b$.
	Further, let $\Delta$ be the complex which is obtained from $\Delta_2$ by adding all edges between any two vertices except $a$ and $b$.

	Note that $\Delta$ satisfies the hypotheses of \Cref{lem:2dim}, so it is sufficient to show that the product on 
	$H_*(K_{\KK[\Delta]})$ is trivial if and only if $\Char \KK$ divides $N$.
	We will use \Cref{prop:hom} for this.
	For dimension reasons, $\phi_i^{I,J} = 0$ for all $i \geq 3$ and all $I,J$.
	Further, the $1$-skeleton of $\Delta$ is complete except for the missing edge from $a$ to $b$, so it is chordal.
	Hence $\phi_1^{I,J} = 0$ for any $I$ and $J$ by \Cref{prop:chordal}.
	To finish the proof, we are going to show that
	\begin{enumerate}[(a)]
		\item for $I := \set{a,b}$ and $J := V(\Delta) \setminus I$, it holds that $\phi_2^{I,J}$ is zero if and only if $\Char \KK \in T$, and
		\item further, $\phi_2^{I',J'} = 0$ independently of the field for any two disjoint vertex sets $I', J'$ with $\set{I', J'} \neq \set{I,J}$.
	\end{enumerate}
	
	We start by showing the first item, so let $I = \set{a,b}$ and let $J = V(\Delta) \setminus I$.
	Let $\sigma_1, \sigma_2 \in C_2(\Delta;\KK)$ be the sums of all triangles in the first and the second copy of $\Delta_1$, respectively, endowed with suitable signs.
	Clearly, the boundary of both $\sigma_1$ and $\sigma_2$ is $N$ times $\gamma$.
	So $\sigma_1$ and $\sigma_2$ are both cycles if $\Char \KK$ divides $N$. 
	Further, their difference $\sigma :=\sigma_1-\sigma_2$ is a cycle, independently of the field.

	If $\Char \KK \in T$, then $\Hi{2}{\Delta}$ is generated by (the classes of) $\sigma_1$ and $\sigma_2$.
	The image of $\sigma_1$ (resp. $\sigma_2$) in $\restr{\Delta}{I} * \restr{\Delta}{J}$ is supported on the contractible subcomplex $\set{a} * \restr{\Delta}{J}$ (respectively $\set{b} * \restr{\Delta}{J}$), so $\phi_2^{I,J}$ sends it to zero.
	Thus, if $\Char \KK \in T$ then $\phi_2^{I,J}$ is the zero map for the given choice of $I$ and $J$.

	Now assume that $\Char \KK \notin T$.
	Note that $\Hi{2}{\Delta} = \Hi{2}{\Delta_2}$, because the two complexes differ only in their $1$-skeleton.
	So the map $\phi_2^{I,J}$ can be factored as
	\[
		\Hi{2}{\Delta} \to \Hi{2}{\Delta_2} \stackrel{\psi}{\to} \Hi{2}{\restr{\Delta_2}{I} * \restr{\Delta_2}{J}} \to \Hi{2}{\restr{\Delta}{I} * \restr{\Delta}{J}}.
	\]
	The first map is a isomorphism and the last map is injective, because only $1$- and $2$-cells are added, so no $2$-cycle can become a boundary.
	Hence $\phi_2^{I,J}$ is nonzero if and only if the middle map $\psi$ is nonzero.
	
	Note that $\restr{\Delta_2}{J}$ retracts onto $\gamma$, hence  $\restr{\Delta_2}{I} * \restr{\Delta_2}{J}$ retracts onto $\restr{\Delta_2}{I} * \gamma$, which is just a suspension of $\gamma$, i.e., a $2$-sphere.
	This implies that a generator $\tau$ for $\Hi{2}{\restr{\Delta_2}{I} * \restr{\Delta_2}{J}}$ is given by a signed sum of all $2$-faces in $\restr{\Delta_2}{I} * \restr{\Delta_2}{J}$ containing one of the edges of $\gamma$ and either $a$ or $b$.
	On the other hand, $\Hi{2}{\Delta}$ is generated by (the class of) $\sigma := \sigma_1-\sigma_2$.
	By construction, it winds $N$ times around $\gamma$.
	Thus, under the above mentioned retraction it will be mapped to $\pm N$ times the generator of $\Hi{2}{\restr{\Delta_2}{I} * \gamma}$ and hence $\psi(\sigma) = \pm N \cdot \tau \neq 0$.
	In conclusion, $\phi_2^{I,J}$ is zero if and only if $\Char \KK \in T$.

	Finally, we show that  $\phi_2^{I',J'} = 0$ for any two disjoint vertex sets $I', J'$ with $\set{I', J'} \neq \set{I,J}$.
	By the K\"unneth formula, $\Hi{2}{\restr{\Delta}{I'} * \restr{\Delta}{J'}}$ can be nontrivial only if either $\Hi{0}{\restr{\Delta}{I'}} \neq 0$ or $\Hi{0}{\restr{\Delta}{J'}} \neq 0$.
	Hence we only need to consider the case $I' = I = \set{a,b}$. 
	
	It remains to show that $\phi_2^{I,J'} = 0$ if $J' \subsetneq V(\Delta) \setminus \set{a,b}$. 
	Let $\omega \in C_2(\restr{\Delta}{I \cup J'};\KK)$ be a cycle.
	If $\omega$ contains a $2$-face $F$, then for every edge of $F$ it also contains another $2$-face with that edge.
	Using the definition of $\Delta$, it follows that for each copy of $\Delta_1$ inside $\Delta$, $\omega$ contains either all or none of its $2$-faces.
	As $I \cup J' \neq V(\Delta)$, $\omega$ cannot contain all $2$-faces of $\Delta$, so all triangles in $\omega$ lie in the same copy of $\Delta_1$, say in the one with vertex $a$.
	But this implies that the image of $\omega$ in $C_2(\restr{\Delta}{I} * \restr{\Delta}{J'};\KK)$ it is supported on the contractible subcomplex $\set{a} * \restr{\Delta}{J'}$, so as above $\phi_2^{I,J'}$ sends it to zero.

\medskip

	(2) We proceed similar as in the proof of part (1).
	$\Delta_1$ is still the simplicial complex indicated in Fig.~\ref{fig:ex1}.
	Let $\Gamma$ be the complex obtained from $\Delta_1$ by adding all edges which do not contain $a$.
	
	As above, $\Gamma$ satisfies the hypotheses of \Cref{lem:2dim} and its $1$-skeleton is chordal,
	so we only need to consider the maps $\phi_2^{I,J}$.
	Let $J \subset V(\Gamma)$ be the set of neighbors of $a$ and let $I := V(\Gamma) \setminus J$.
	Here, a \emph{neighbor} of $a$ is any vertex sharing an edge with it. 
	This time, we are going to show the following:
	\begin{enumerate}[(a)]
		\item The map $\phi_2^{I,J}$ is nonzero if and only if $\Char \KK \in T$.
		\item Further, $\phi_2^{I',J'} = 0$ if $\Char \KK \notin T$ for any two disjoint vertex sets $I', J'$ with $\set{I', J'} \neq \set{I,J}$.
	\end{enumerate}
	
	We show both items simultaneously.
	Our description of $\Delta_1$ given above shows that
	\[ \Hi{2}{\Gamma} = \Hi{2}{\Delta_1} = \begin{cases}
		\KK & \text{ if } \Char \KK \in T\\
		0 	& \text{ otherwise.}
	\end{cases}\]
	So clearly $\phi_2^{I',J'} = 0$ if $\Char \KK \notin T$ for any two disjoint vertex sets $I', J'$ with $I' \cup J' = V(\Gamma)$, in particular for $I' = I$ and $J' = J$.
	Also, for two disjoint vertex sets $I', J'$ with $I' \cup J' \subsetneq V(\Gamma)$, one shows similarly to the argument in the proof of part (1) that $\Hi{2}{\restr{\Gamma}{I' \cup J'}} = 0$, so the corresponding map vanishes.
	
	It remains to show that $\phi_2^{I,J} \neq 0$ if $\Char \KK \in T$.
	In this case, a generator of $\Hi{2}{\Gamma}$ is given by the sum $\sigma_1 \in C_2(\Gamma;\KK)$ of all $2$-faces of $\Gamma$, endowed suitable signs.
	Now, $\sigma_1$ contains a $2$-face $F$ which contains $a$.
	Further, $F \cap I = \set{a} \in \restr{\Gamma}{I}$ is isolated and $F \cap J \in \restr{\Gamma}{J}$ is a facet, because $\restr{\Gamma}{J}$ is one-dimensional.
	Hence $\phi_2^{I,J}(\sigma_1) \neq 0$ by \Cref{lem:complete}
\end{proof}

\begin{example}
Let us illustrate the construction of part (1) in the case $T = \set{2}$.
In this case, $\Delta_1$ is a real projective plane.
One can think of it as being obtained by gluing the star of $a$, which is a disc, on the boundary of a M\"obius strip, which is the lower part of $\Delta_1$ in \Cref{fig:ex1}.
Now $\Delta_2$ is obtained by gluing two copies of this projective planes along $\gamma$.
In this situation $\sigma_1$ and $\sigma_2$ are the fundamental classes (over $\ZZ/2$)  of the two projective planes.

The map $\Delta_2 \to \restr{\Delta_2}{I} * \restr{\Delta_2}{J} \to \restr{\Delta_2}{I} * \gamma \approx S^2$ corresponds to retracting the M\"obius strips to $\gamma$.
Under this map $\sigma_1$ and $\sigma_2$ get deformed to ``{two times}'' the lower resp. upper hemisphere of the resulting sphere. Hence $\sigma_1 - \sigma_2$ is mapped to two times the fundamental class $\tau$ of $S^2$. 
\end{example}

It seems a natural question if the finiteness of the set $T$ in \Cref{thm:main} is really necessary.
In other words, could there be a simplicial complex $\Delta$ such that $\KK[\Delta]$ is Golod in infinitely many characteristics and non-Golod in infinitely many other characteristics?
Indeed, such a phenomenon is excluded by the following result.
For completeness, we also show that the Golod property only depends on the characteristic.
\finiteness

\begin{proof}
	The claim follows almost immediately from a characterization of the Golod property given by Berglund in \cite[Theorem 3]{Berglund2006}.
	We briefly recall this characterization.
	Let $I := I_\Delta$ be the Stanley-Reisner ideal of $\Delta$ and let $M$ be a minimal set of monomial generators of it. Note that $M$ only depends on $\Delta$ and not on the field $\KK$.
	Recall that the \emph{lcm-lattice} of $I$, $L_I$, is the lattice of all least common multiples of subsets of $M$, ordered by divisibility.
	For any monomial $m \in L_I$, let $M({\leq m})$ denote the set of those monomials in $M$ which divide $m$.
	Moreover, for any finite set $N$ of monomials, let $m_N$ denote the least common multiple of all elements in $N$.
	In \cite{Berglund2006}, Berglund associates to each finite set $N$ of monomials a natural number $c(N)$ and a lattice of certain subsets of $N$ which is denoted by $K(N)$.
	The exact definitions are not relevant here; we only need that these constructions do not depend on the underlying field $\KK$.
	Finally, for a finite poset $P$ we set
	\[ \tilde{H}(P; \KK)(z) := \sum_{i \geq -1} \dim_\KK(\tilde{H}_i(P; \KK))z^i, \]
	where $\tilde{H}^i(P; \KK)$ is the homology of the order complex of $P$.
	Theorem 3 of \cite{Berglund2006} asserts that $\KK[\Delta]$ is Golod if and only if for each $m \in L_I$, it holds that
	\begin{equation}\label{eq:thm3}
		\tilde{H}((\hat{0}, m)_{L_I}; \KK)(z) = \sum_{\substack{S \in K(M({\leq m}))\\ m_S = m}} (-z)^{c(S)-1}\tilde{H}((\hat{0}, S)_{K(M({\leq m}))}; \KK)(z),
	\end{equation}
	where $(\hat{0}, m)_{L_I}$ denotes the lower interval defined by $m$ in $L_I$, and similarly $(\hat{0}, S)_{K(M({\leq m}))}$ is the lower interval defined by $S$ in $K(M({\leq m}))$.

	Now we  observe that both sides of \eqref{eq:thm3} depend on $\KK$ only via the homology of some simplicial complexes.
	But for any simplicial complex $\Gamma$, $\tilde{H}_i(\Gamma; \KK)$ depends only on the characteristic of $\KK$.
	Moreover, it holds that 
	\[ \dim_{\FF_p} \tilde{H}_i(\Gamma; \FF_p) =  \dim_{\QQ} \tilde{H}_i(\Gamma; \QQ) \]
	 for all but finitely many primes $p$.
	Thus, the set of characteristics in which \eqref{eq:thm3} holds is either finite or it has a finite complement, and these two cases can be distinguished by considering it over $\QQ$.
\end{proof}

\section{Two explicit examples}\label{sec:example}
\newcommand{\frh}{f_{\mathrm{RH}}}
\newcommand{\flh}{f_{\mathrm{LH}}}
We give two explicit examples of simplicial complexes which show the phenomena observed in the last section.
In particular, we evaluate the two sides of \eqref{eq:ps} in these two examples.
For this, we define
\[ \frh := \sum_{i\geq 1}\sum_{j \geq 0} \dim_\KK (\Tor^S_i(A,\KK)_j) t^{i} x^j, \]
so the right-hand side of \eqref{eq:ps} is given by
\[ \frac{(1+t x)^n}{1-t \frh}. \]
The polynomial $\frh$ can easily be determined from the graded Betti numbers of $A$, which we computed in our examples with \texttt{Macaulay2} \cite{M2}.
For the left-hand side of \eqref{eq:ps}, it follows from Theorem 1 of \cite{Berglund2006} (see also \cite{back}) that it is also of the form
\[ \frac{(1+t x)^n}{1-t \flh} \]
with some polynomial $\flh$ in $t$ and $x$.
To compute it in our examples, we evaluated the formula (9) of \cite{Berglund2006} using a computer.

\begin{example}
Let $\Delta$ be the simplicial complex with the facets
\begin{center}\begin{tabular}{ccccc}
	$124$&$235$&$341$&$452$&$513$ \\
	$12a$&$23a$&$34a$&$45a$&$51a$ \\
	$12b$&$23b$&$34b$&$45b$&$51b$
\end{tabular}\end{center}
on the vertex set $V = \set{1,2,3,4,5,a,b}$.
The generators of the corresponding Stanley-Reisner ideal are
\begin{center}\begin{tabular}{ccccc}
$x_1x_2x_3$& $x_2x_3x_4$& $x_3x_4x_5$& $x_4x_5x_1$& $x_5x_1x_2$\\
$x_1x_3x_a$& $x_1x_4x_a$& $x_2x_4x_a$& $x_2x_5x_a$& $x_5x_3x_a$\\
$x_1x_3x_b$& $x_1x_4x_b$& $x_2x_4x_b$& $x_2x_5x_b$& $x_5x_3x_b$\\
$x_ax_b$ 
\end{tabular}\end{center}
Geometrically, $\Delta$ is a M\"obius strip with two $2$-balls glued along its boundary.
Using a similar argument as in the proof of part (1) of \Cref{thm:main}, one shows that $\KK[\Delta]$ is Golod if and only if $\Char \KK = 2$,
where the ``critical'' sets are $I = \set{a,b}$ and $J = \set{1,\dotsc,5}$.
Note that the $1$-skeleton of $\Delta$ is already chordal, so we do not need to add additional edges as in the proof of \Cref{thm:main}.
For $\KK = \QQ$, we computed that
\begin{align*}
\frh &= (x^2 + 15x^3)t + 35x^4t^2 + 26x^5t^3 + (5x^6 + x^7)t^4,  \\
\flh &= (x^2 + 15x^3)t + 35x^4t^2 + 26x^5t^3 + 5x^6t^4 - x^7t^5.
\end{align*}
As expected, there is a strict inequality in \eqref{eq:ps}.
For $\KK = \FF_2$, the ring is Golod, so the equality in \eqref{eq:ps} is attained. In this case, it holds that
\begin{align*}
\frh = \flh = (x^2 + 15x^3)t + 35x^4t^2 + (26x^5+2x^6)t^3 + (7x^6 + 2x^7)t^4+x^7t^5
\end{align*}
Note that in this example both sides of \eqref{eq:ps} depend in the field.
\end{example}

\begin{example}
Let $\Delta$ be the simplicial complex with the facets
\begin{center}\begin{tabular}{ccccc}
	$124$&$235$&$341$&$452$&$513$ \\
	$12a$&$23a$&$34a$&$45a$&  \\
	$51b$&$1ab$&$a5b$& & 
\end{tabular}\end{center}
on the vertex set $V = \set{1,2,3,4,5,a,b}$.
The generators of the corresponding Stanley-Reisner ideal are
\begin{center}\begin{tabular}{ccccc}
$x_1x_2x_3$& $x_2x_3x_4$& $x_3x_4x_5$& $x_4x_5x_1$& $x_5x_1x_2$\\
$x_1x_3x_a$& $x_1x_4x_a$& $x_2x_4x_a$& $x_2x_5x_a$& $x_5x_3x_a$\\
$x_5x_1x_a$& $x_2x_b$& $x_3x_b$& $x_4x_b$ 
\end{tabular}\end{center}
This is a triangulation of the real projective plane, which is obtained from the usual $6$-vertex triangulation by subdividing the $2$-cell $51a$.
It can be shown similar as in the proof of part (2) of \Cref{thm:main} that $\KK[\Delta]$ is Golod if and only if $\Char \KK \neq 2$.
Here, the ``critical'' sets are $I = \set{2,3,4,b}$ and $J = \set{5,1,a}$.

Again, we provide some numerical data for completeness.
For $\KK = \QQ$, the ring is Golod and it holds that
\begin{align*}
\frh = \flh = (3x^2 + 11x^3)t + (3x^3 + 28x^4)t^2 + (x^4 + 24x^5)t^3 + 7x^6t^4.
\end{align*}
For $\KK = \FF_2$, it holds that
\begin{align*}
\frh = (3x^2 + 11x^3)t + (3x^3 + 28x^4)t^2 + (x^4 + 24x^5)t^3 + (7x^6+x^7)t^4 + x^7t^5,
\end{align*}
while $\flh$ is the same as over $\QQ$.
So in this example, somewhat surprisingly, only one side of \eqref{eq:ps} depends on the field.
\end{example}

Apart from the equation \eqref{eq:ps}, the triviality of the Massey products on $H_*(K_{\KK[\Delta]})$ and \cite[Theorem 3]{Berglund2006}, there is a further well-known characterizations of Golod rings.
Namely, a ring $A$ is Golod if and only if its homotopy Lie algebra $\pi^{\geq 2}(A)$ is a free (graded) Lie algebra, cf. Chapter 10 of \cite{Avra}.
It might be instructive to study the structure of the homotopy Lie algebra in these examples.

\section{Concluding remarks}

\subsection{Skeleta and higher chordality}
In this section, we give two consequences of \Cref{lem:complete} which we consider to be of independent interest.
In the sequel, $\Delta$ is always a simplicial complex.
The first one is the following corollary.
\begin{corollary}
	The map $\phi_i^{I,J}$ depends only on the $i$-skeleton of $\Delta$.
\end{corollary}
\begin{proof}
	It is clear that the map depends only on the $(i+1)$-skeleton of $\Delta$.
	Adding $(i+1)$-dimensional simplices to $\Delta$ only turns complete $i$-cycles into boundaries.
	By \Cref{lem:complete}, all complete $i$-cycles lie in the kernel of $\phi_i^{I,J}$, so this does not affect the map.
\end{proof}

\begin{remark}
	Based on the preceding corollary, one might be tempted to conjecture that the map $\phi_i^{I,J}$ depends only in the \emph{pure} $i$-skeleton of $\Delta$. But this is false by the following counterexample.
	
	Consider the join of an empty triangle with an $S^0$. This complex is not Golod, as it is a join (or Gorenstein*).
	But the complex $\Delta$ obtained by adding an edge between the two "poles" is Golod: If the two poles are $I$ and the triangle is $J$, then $\restr{\Delta}{I}$ is contractible, so $\restr{\Delta}{I} * \restr{\Delta}{J}$ is as contractible as well.
	All proper restrictions of $\Delta$ have no second homology, and the $1$-skeleton is chordal.
	Hence $\Delta$ is Golod.
\end{remark}

Adiprasito, Nevo and Samper define in \cite{ANS} several high-dimensional extensions of the notion of a chordal graph.
In particular, they define a simplicial complex to be \emph{decomposition $k$-chordal}, if every $k$-cycle $z$ can be written as a sum of complete $k$-cycles $(z_i)$, such the vertices of each $(z_i)$ are also vertices of $z$.
The sufficiency of \Cref{prop:chordal} extends to this setting:
\begin{proposition}
	If $\Delta$ is a decomposition $k$-chordal simplicial complex,
	then $\phi_k^{I,J} = 0$ for all non-empty disjoint subsets $I,J \subset V$.
	In particular, if $\Delta$ is decomposition $k$-chordal for all $k$, then the product on $H_*(K_{\KK[\Delta]})$ is trivial.
\end{proposition}
\begin{proof}
	This is immediate from \Cref{lem:complete}.
\end{proof}

\subsection{Degree bounds}

By \Cref{prop:hom}, $\KK[\Delta]$ is Golod if the maps $\phi_i^{I,J}$ vanish for all $i \in \NN$.
It is clear that one only has to consider $i \leq \dim \Delta$.
Moreover, it is in fact sufficient to consider $i \leq \reg \KK[\Delta] - 1$, where $\reg \KK[\Delta]$ denotes the Castelnuovo-Mumford regularity.
This is immediate from Hochster's formula, which implies that 
\[ \reg \KK[\Delta] = \max\set{j \with \Hi{j-1}{\restr{\Delta}{I}} \neq 0 \text{ for some }I \subseteq V}. \]
It follows from \cite[Theorem 6.5]{BJ07} that if $\Delta$ is flag, then the vanishing of $\phi_1^{I,J}$ (for all $I,J$) is sufficient for $\KK[\Delta]$ to be Golod.
An optimistic generalization of this would be the following assertion:
If $\Delta$ has no minimal non-faces of dimension $\geq k$ and $\phi_i^{I,J}=0$ for all $i \leq k$, then $\KK[\Delta]$ is Golod.
However, this can easily seen to be false. An easy example is the join of two boundaries of $(k-1)$-simplices.
Instead we ask we following question, which is motivated by an analogous result in \cite{ANS}:
\begin{question}
Let $\Delta$ be a simplicial complex with vertex set $V$.
Assume that $\Delta$ has no minimal non-faces of dimension $\geq k$,
and assume further that $\phi_i^{I,J} = 0$ for all $i \leq 2k-1$ and all non-empty disjoint subsets $I,J \subseteq V$.
Is $\KK[\Delta]$ Golod?
\end{question}

\section*{Acknowledgment}
The author thanks Sean Tilson for many inspiring discussions and Volkmar Welker for suggesting the topic of the present note.
Moreover, the author wishes to thank Yi-Huang Shen and the anonymous referee for several helpful comments.
\bibliographystyle{alpha}
\bibliography{Golod}

\end{document}